\documentclass[11pt]{amsart}

\usepackage{amssymb}
\usepackage{amsthm}
\usepackage{amsmath}
\usepackage{mathrsfs}
\usepackage{mathtools}
\usepackage{amsbsy}
\usepackage{stmaryrd}
\usepackage{trimclip}
\usepackage[all]{xy}
\usepackage{bm}
\usepackage{hyperref}
\usepackage{tikz}
\usepackage{array}
\usepackage{float}
\usepackage{enumerate}
\usepackage{xcolor}
\usepackage{hhline}
\allowdisplaybreaks
\usepackage[noadjust]{cite}
\usepackage[normalem]{ulem}

\usepackage{geometry}
\usepackage{standalone}
\usepackage{tikz}
\usetikzlibrary{matrix, decorations.pathreplacing}
\usetikzlibrary{shapes.geometric, positioning}
\usetikzlibrary{angles}
\usepackage{amsmath}
\usetikzlibrary{calc}
\usetikzlibrary{arrows}

\usepackage{caption}
\usepackage{subcaption}
\usepackage{tabu}
\usepackage{diagbox}
\usepackage{tikz}
\usepackage{bbm}
\usepackage{booktabs}

\usepackage[noabbrev,capitalise]{cleveref}

\newenvironment{enumerate*}%
{\begin{enumerate}[(I)]%
\setlength{\itemsep}{10pt}%
\setlength{\parskip}{0pt}}%
{\end{enumerate}}

\newtheorem{theorem}{Theorem}[section]
\newtheorem{proposition}[theorem]{Proposition}
\newtheorem{corollary}[theorem]{Corollary}

\newtheorem{question}[theorem]{Question}
\newtheorem{problem}[theorem]{Problem}
\newtheorem{lemma}[theorem]{Lemma}

\theoremstyle{definition}
\newtheorem{definition}[theorem]{Definition}

\DeclareMathOperator{\Span}{span}

\DeclareMathOperator{\GL}{GL}
\DeclareMathOperator{\SL}{SL}

\DeclareMathOperator{\codim}{codim}
\DeclareMathOperator{\fl}{Fl}
\DeclareMathOperator{\cost}{cost}

\newcommand{\R}{\ensuremath{\mathbb R}} 
\newcommand{\N}{\ensuremath{\mathbb N}} 
\newcommand{\F}{\ensuremath{\mathbb F}} 
\newcommand{\K}{\ensuremath{\mathbb K}} 
\newcommand{\defeq}{\ensuremath{\coloneqq}}
\newcommand{\floor}[1]{\ensuremath{\left\lfloor#1\right\rfloor}}
\newcommand{\ceil}[1]{\ensuremath{\left\lceil#1\right\rceil}}
\newcommand{\abs}[1]{\ensuremath{\lvert#1\rvert}}

\newcommand{\symmetric}[1]{\mathfrak{S}_{#1}}

\newcommand{\z}{\hspace{0.5pt}}

\NewDocumentCommand\flag{O{}mo}{%
\IfNoValueTF{#3}
{
\ensuremath{{{#2}}_{#1}^{\bullet}}
}
{\ensuremath{{#2}_{#1}^{#3}}}%
}

\makeatletter
\DeclareRobustCommand{\shortto}{%
\mathrel{\mathpalette\short@to\relax}%
}

\newcommand{\short@to}[2]{%
\mkern2mu
\clipbox{{.5\width} 0 0 0}{$\m@th#1\vphantom{+}{\shortrightarrow}$}%
}
\makeatother

\newcommand{\perm}[3][]{\ensuremath{{\sigma}_{\flag{#2}, \flag{#3}}^{#1}}}

\title{Simultaneous generating sets for flags}
\author{Federico Glaudo, Noah Kravitz, Chayim Lowen}

\address[]{School of Mathematics, Institute for Advanced Study, 1 Einstein Dr., Princeton, NJ 08540, USA}
\email{federico.glaudo@ias.edu}

\address[]{Department of Mathematics, Princeton University, Princeton, NJ 08540, USA}
\email{nkravitz@princeton.edu}

\address[]{Department of Mathematics, Princeton University, Princeton, NJ 08540, USA}
\email{chayiml@princeton.edu}

\begin{document}

\begin{abstract}
We prove that any triple of complete flags in $\R^d$ admits a common generating set of size $\floor{5d/3}$ and that this bound is sharp.
This result extends the classical linear-algebraic fact---a consequence of the Bruhat decomposition of $\GL_d(\R)$---that any pair of complete flags in $\R^d$ admits a common generating set of size $d$.  We also deduce an analogue for $m$-tuples of flags with $m>3$.
\end{abstract}

\maketitle

\section{Introduction}
A \emph{complete flag} in $\R^d$ is a sequence $\flag{U}$ of nested subspaces $$\{0\}=\flag{U}[0] \subset \flag{U}[1] \subset \cdots \subset \flag{U}[d]=\mathbb R^d,$$ where $\dim \flag{U}[i]=i$ for each $i$.  We say that a set $S$ of vectors \emph{generates} $\flag{U}$ if each $\flag{U}[i]$ is the span of some subset of $S$.  Equivalently, $S$ generates $\flag U$ if $S$ contains a vector in $\flag{U}[i] \setminus \flag{U}[i-1]$ for each $i \in [d]$.  A remarkable linear-algebraic fact, essentially equivalent to the Bruhat cell decomposition of $\GL_d(\mathbb{R})$, states that any two complete flags in $\mathbb{R}^d$ can be simultaneously generated by a single set of size $d$.
The aim of this paper is to study the analogous problem for three or more flags.

For flags $\flag[1]{U}, \ldots, \flag[m]{U}$ in $\mathbb{R}^d$, let
$\mu(\flag[1]{U}, \ldots, \flag[m]{U})$
denote the smallest possible size of a set that simultaneously generates 
each of
$\flag[1]{U}, \ldots, \flag[m]{U}$.  Let $\mu(m,d)$ denote the maximum value of $\mu(\flag[1]{U}, \ldots, \flag[m]{U})$ over all 
complete
flags $\flag[1]{U}, \ldots, \flag[m]{U}$ in $\mathbb{R}^d$.  In this language, the linear-algebraic fact from the previous paragraph says that $\mu(2,d)=d$ for all $d \in \N$.  Our main result is an exact determination of $\mu(m,d)$ for all $m,d$.

\begin{theorem}\label{thm:main}
For $m,d \in \N$, we have
\[
\mu(m, d) =
\begin{cases}
d &\quad \text{if $m = 1$ or $d = 1$};\\
\frac{md}2 &\quad \text{if $m \geq 2$ is even, $d \geq 2$};\\
\frac{md}2 + \left\lfloor \frac{2d}3\right\rfloor -\frac d2 &\quad \text{if $m \geq 3$ is odd, $d \geq 2$}.
\end{cases}
\]
\end{theorem}
The first case of \cref{thm:main} is trivial.  The second case follows quickly from $\mu(2,d)=d$.  The third case is the main work of this paper, and the subcase $m=3$ turns out to be the heart of the matter.  The 
value ${\mu(3, d) = \lfloor 5d/3 \rfloor}$ is \emph{a priori} surprising since, as we discuss 
later,
a naive calculation based on generic triples of flags might lead one to guess
$\mu(3,d)=\lceil 3d/2 \rceil$.

We work over the field $\R$ throughout for simplicity. Nevertheless, both the statement and the proof of \cref{thm:main} continue to hold with $\mathbb{R}$ replaced by any other infinite field.

\subsection{Organization of the paper} We describe the situation for a pair flags in \cref{sec:2flags}.  We then analyze the behavior of generic tuples of flags in \cref{sec:generic}.  In \cref{sec:superadditivity} we show how to deduce \cref{thm:main} from the upper bound $\mu(3, d) \leq 5d/3$.  In \cref{sec:3-flags}---the most substantial part of the paper---we prove this upper bound.
We conclude with a few remarks and open questions in \cref{sec:concluding}.

\subsection{Notation and terminology}
We will omit the adjective ``complete'' in ``complete flag''.
We say that the subspace $\flag{U}[i]$ of a flag $\flag{U}$ constitutes the \emph{$i$-th layer} of $\flag{U}$.  A vector $s$ is \emph{new} for a layer $\flag{U}[i]$ if $s \in \flag{U}[i] \setminus \flag{U}[i-1]$.  Thus a set $S$ generates a flag $\flag{U}$ if it contains a new vector for each layer of $\flag{U}$.  A \emph{simultaneous generating set} for a tuple of flags is a set which generates each of the flags individually. We will sometimes drop the adjective ``simultaneous''.

The set of natural numbers is $\N:=\{1,2,\ldots \}$.  We write $[d]:=\{1,2,\ldots, d\}$ for $d \in \N$.
The group of permutations of $[d]$ is denoted $\symmetric{d}$.

\section{Two flags}\label{sec:2flags}

Our story begins with the $m=2$ case of \cref{thm:main}.  Beyond the identity $\mu(2,d)=d$, one can precisely describe how vectors are reused across a given pair of flags.

\begin{theorem}\label{thm:2flags}
Let $\flag{U},\z\flag{V}$ be flags in $\mathbb{R}^d$.  Then there exist a basis $s_1, \ldots, s_d$ of $\mathbb{R}^d$ and a permutation $\sigma \in \symmetric{d}$ such that each $s_i$ is new for both 
$U_i$ and $V_{\sigma(i)}$. It follows that
\[
\flag{U}[i]=\Span(\{s_1, \ldots, s_i\}) \quad \text{and} \quad \flag{V}[j]=\Span(\{s_{\sigma^{-1}(1)}, \ldots, s_{\sigma^{-1} (j)}\}) \quad \text{for all $i, j \in [d]$}.
\]
Moreover, the permutation $\sigma$ is uniquely determined by $\flag{U},\z\flag{V}$.
\end{theorem}

This theorem is a combinatorial restatement of the celebrated \emph{Bruhat cell decomposition} of $\GL_d(\mathbb{R})$, which in turn 
boils down to Gaussian elimination.
We will be content to sketch the link here; we refer the reader to~\cite[Chapter IV.2]{bourbaki}
for additional context.
Let $B \subset \GL_d(\mathbb{R})$ be the Borel subgroup consisting of the upper-triangular matrices.
Identify each $\sigma \in \symmetric{d}$ with the corresponding $d \times d$ permutation matrix.  
The Bruhat decomposition states that $\GL_d(\mathbb{R})$ is the disjoint union, over $\sigma \in \symmetric{d}$, of the double cosets $B \sigma B$.

The space of all flags in $\R^d$, denoted $\fl_d(\R)$, is naturally identified with the quotient $\GL_d(\R)/B$ (by letting the $i$-th column of the matrix be some new vector for the
$i$-th layer of the corresponding flag).
The Bruhat decomposition of $\GL_d(\R)$ naturally induces a decomposition of $\fl_d(\R)$ indexed by $\sigma \in \symmetric{d}$.
To see the connection with \cref{thm:2flags}, first note that, by a change of basis, it suffices to consider the case where $\flag{U}$ is the standard flag in $\R^d$ (in which the $i$-th standard basis vector is new for the $i$-th layer).
Under the correspondence $\fl_d(\R) \simeq \GL_d(\R)/B$, the flag $\flag V$ is given by some 
coset $gB$ (with $g \in \GL_d(\R)$).
The Bruhat decomposition provides $b \in B$
and (unique) $\sigma \in \symmetric{d}$
such that
$gB = b\sigma B$.
Then the conclusion of \cref{thm:2flags} holds with $s_i$ being the $i$-th column of $b$.

For the sake of completeness---and in order to introduce a framework which will be useful later---we present a self-contained combinatorial proof of \cref{thm:2flags}.
The following argument is based on the one given by user Eric Wofsey on Math Stack Exchange~\cite{Wofsey}.

\begin{proof}[Proof of \cref{thm:2flags}]
For each $i \in [d]$, let $\sigma(i)$ be the smallest index $j$ for which
\[
(\flag{U}[i] \setminus \flag{U}[i - 1]) \cap \flag{V}[j] \neq \varnothing.
\]
We claim that $\sigma$ is a permutation of $[d]$. For this, it suffices to show that ${\lvert \sigma^{-1}([k])\rvert =k}$ for each $k \in [d]$.  Observe that $\sigma(i) \leq k$ if and only if
\[
(\flag{U}[i] \setminus \flag{U}[i-1]) \cap \flag{V}[k] \neq \varnothing,
\]
or, equivalently, if
\[
\dim(\flag{U}[i] \cap \flag{V}[k])
>\dim(\flag{U}[i - 1] \cap \flag{V}[k]).
\]
As $i$ increases from $0$ to $d$, the quantity $\dim(\flag{U}[i] \cap \flag{V}[k])$ grows from $0$ to $k$, increasing by at most 1 at each step.
So there are precisely $k$ indices $i$ with $\dim(\flag{U}[i] \cap \flag{V}[k])>\dim(\flag{U}[i-1] \cap \flag{V}[k])$. This proves the claim.

The definition of $\sigma$ ensures that for each $i \in [d]$, we have
\[
(\flag{U}[i] \setminus \flag{U}[i-1]) \cap (\flag{V}[\sigma(i)] \setminus \flag{V}[\sigma(i)-1]) \neq \varnothing.
\]
For each $i$, we choose a vector $s_i$ in this intersection. This establishes the existence part of the theorem.

For uniqueness of the permutation, suppose that $\tilde\sigma \in \symmetric{d}$ also satisfies the conclusion of the theorem. Then for each $i \in [d]$ we have
\[
(\flag{U}[i] \setminus \flag{U}[i - 1]) \cap (\flag{V}[\tilde\sigma(i)] \setminus \flag{V}[\tilde\sigma(i) - 1]) \not=\varnothing,
\]
and the definition of $\sigma$ implies that $\tilde\sigma(i)\ge \sigma(i)$. Since $\sigma,\tilde \sigma$ are both permutations, they must be equal.
\end{proof}
Several aspects of \cref{thm:2flags} will play prominent roles later in the paper.
The primary takeaway is the centrality of the permutation $\sigma \in \symmetric{d}$ appearing in the conclusion of the theorem.
We will write $\perm{U}{V}$ for the permutation thereby associated to the (ordered) pair of flags $\flag{U}$, $\flag{V}$.  Note that $\perm{V}{U}={\perm[-1]{U}{V}}$.

We take this opportunity to highlight a qualitative difference between the case of two flags and the case of more than two flags.  
In the former case, \cref{thm:2flags} assures us that there is a unique optimal way (encoded by $\sigma_{\flag{U},\flag{V}}$) to reuse new vectors across the flags.
This rigidity does not extend to the latter case, where in general there are many different optimal ways to reuse new vectors.
This additional flexibility gives the setting of three or more flags a ``combinatorial'' flavor, which contrasts with the more ``algebraic'' flavor of
\cref{thm:2flags}.  

\section{Generic flags}\label{sec:generic}

A natural starting point for our investigation into the main problem is the setting where the flags are chosen ``generically'', in the sense of the following definition.
\begin{definition}
We say that subspaces $E_1, \dots, E_m$
of $\R^d$ are \emph{transverse}
if they satisfy 
\[
\codim \left(\bigcap_{i \in S} E_i\right) = \min\left(d,\; \sum_{i \in S} \codim E_i \right) \quad \text{for all $S \subseteq [m]$}.
\]
We say that flags $\flag[1]{U}, \dots, \flag[m]{U}$ are \emph{transverse} if for each $m$-tuple of indices 
${(i_1, \ldots, i_m) \in [d]^m}$, the subspaces
$\flag[1]{U}[i_1], \dots, \flag[m]{U}[i_m]$
are transverse.
\end{definition}
Generically-chosen flags are transverse, in the sense that the set of $m$-tuples of transverse flags in $\mathbb{R}^d$ forms a dense open subset of $\fl_d(\R)^m$, the $m$-fold product of flag varieties.

A tuple of flags is transverse when intersections involving their layers are as small as possible.  One might naively expect that transverse $m$-tuples of flags maximize $\mu$ because smaller intersections provide fewer ``opportunities'' for reusing vectors across flags.  This intuition gives the right  answer for even $m$ and the wrong answer for odd~$m$.

\begin{proposition}\label{prop:generic}
Let $m,d \geq 2$.  If $\flag[1]{U}, \ldots, \flag[m]{U}$ are transverse flags in $\mathbb{R}^d$, then $$\mu(\flag[1]{U}, \ldots, \flag[m]{U})=\lceil md/2 \rceil.$$
\end{proposition}

\begin{proof}
We begin with the lower bound on $\mu(\flag[1]{U}, \ldots, \flag[m]{U})$.  If $d$ is even, then the lower bound follows immediately from the fact, due to transversality, that the layers $\flag[i]{U}[d/2]$ (for $i \in [m]$) intersect pairwise trivially.  Now suppose that $d=2e+1$ is odd.  Transversality again forces the layers $\flag[i]{U}[e]$ to intersect pairwise trivially. So we already obtain
the bound
\[
\mu(\flag[1]{U}, \ldots, \flag[m]{U}) \geq me=m\lfloor d/2 \rfloor,
\]
by counting only the new vectors for layers of dimension at most $e$.  Transversality prevents a vector from being new for both a layer of dimension ${e+1}$ and a layer of dimension at most $e$, since two such layers from different flags intersect trivially.  Transversality moreover implies that any single vector can be new for at most two different layers of dimension ${e+1}$, since any three layers of dimension ${e+1}$ intersect trivially.
Hence we need at least $\lceil m/2 \rceil$ further new vectors for the layers of dimension ${e+1}$, for a total of
\[
me+\lceil m/2 \rceil=\lceil md/2 \rceil
\]
vectors in any simultaneous generating set.

We turn now to the upper bound. 
As a consequence of the subadditivity property of $\mu$ described in \cref{prop:splitflags}, the upper bound for all $m\ge 2$ follows from the upper bound for $m=2,3$.  The $m=2$ case follows from \cref{thm:2flags}.  

For the $m=3$ case, suppose that $\flag U, \flag V, \flag W$ are transverse flags in $\mathbb{R}^d$.  We observe that if $i+j=d+1$, then $\flag{U}[i] \cap \flag{V}[j]$ has dimension 1 and is contained in neither $\flag{U}[i-1]$ nor $\flag{V}[j-1]$. So we can find a vector
\[
s \in (\flag{U}[i] \setminus \flag{U}[i-1]) \cap (\flag{V}[j] \setminus \flag{V}[j-1]),
\]
i.e., one which is new for both $\flag{U}[i]$ and $\flag{V}[j]$.  The same holds with $(\flag{U}[i],\flag{V}[j])$ replaced by $(\flag{V}[i],\flag{W}[j])$ or $(\flag{W}[i],\flag{U}[j])$.

For each $1 \leq i \leq d/2$, consider the three pairs 
\[
(\flag{U}[i], \flag{V}[d+1-i]), \;\; (\flag{V}[i], \flag{W}[d+1-i]), \;\; (\flag{W}[i], \flag{U}[d+1-i]),
\]
and for each pair choose a common new vector $s$ as above.  If $d$ is odd, then also choose a vector in each of 
\[
(\flag{U}[\ceil{d/2}] \setminus \flag{U}[\floor{d/2}]) \cap (\flag{V}[\ceil{d/2}] \setminus \flag{V}[\floor{d/2}]), \quad \flag{W}[\ceil{d/2}] \setminus \flag{W}[\floor{d/2}].
\]
This gives a collection of $\lceil 3d/2 \rceil$ vectors generating all three flags.
\end{proof}

We remark that the idea behind this upper-bound argument can be extended to a larger family of triples of flags: If $\flag U,\z\flag V,\z\flag W$ are flags in $\mathbb{R}^n$ such that the pair $\flag U,\z\flag V$ is transverse, then ${\mu(\flag U,\flag V, \flag W) \leq \lceil 3d/2 \rceil}$. We leave the details to the interested reader. 

\section{Reducing to the hardest subcase}\label{sec:superadditivity}
As we mentioned earlier, the generic construction does not require the maximum number of generating vectors when $m$ is odd.
It turns out that the missing ingredient in our construction is superadditivity, which will allow us to leverage the ceiling function in \cref{prop:generic}.

\begin{proposition}\label{prop:sum_flags}
Let $m, d, e \in \N$.
Let $\flag[1]{U}, \dots, \flag[m]{U}$ and
$\flag[1]{V}, \dots, \flag[m]{V}$ be $m$-tuples of flags in $\R^{d}$
and $\R^{e}$, respectively.
Then there exists an $m$-tuple of flags
$\flag[1]{W}, \dots, \flag[m]{W}$
in 
$\R^{d + e}$ such that
$$\mu(\flag[1]{W}, \dots, \flag[m]{W}) = \mu(\flag[1]{U}, \dots, \flag[m]{U})+\mu(\flag[1]{V}, \dots, \flag[m]{V}).$$
\end{proposition}
\begin{proof}
For each $i \in [m]$ and $0 \leq j \leq d + e$, define the subspace $\flag[i]{W}[j] \subseteq \mathbb{R}^{d+e}$ by
\[
\flag[i]{W}[j]:= \begin{cases}
\flag[i]{U}[j] \times \{0\}^e &\quad \text{if $j \leq d$};\\
\R^d \times \flag[i]{V}[j - d] &\quad \text{if $j \geq d$}.
\end{cases}
\]
Let $S$ be a simultaneous generating set for $\flag[1]{W}, \dots, \flag[m]{W}$ of size ${|S|=\mu(\flag[1]{W}, \dots, \flag[m]{W})}$.
Let $S_U$ be the intersection of $S$
with the subspace $\R^d \times \{0\}^e \subset \R^{d + e}$,
and let $S_V$ 
be the projection of 
$S \setminus S_U$ onto the factor $\R^e$ of $\R^d \times \R^e = \R^{d + e}$.
Identifying $\R^d \times \{0\}^d$ with $\R^d$ in the obvious way, we see that $S_U$ is a generating set for $\flag[1]{U}, \dots, \flag[m]{U}$. Likewise,
$S_V$ is a generating set for
$\flag[1]{V}, \dots, \flag[m]{V}$. Hence
\[
\mu(\flag[1]{W}, \dots, \flag[m]{W})=\abs{S} \geq \abs{S_U} + \abs{S_V} \geq \mu(\flag[1]{U}, \dots, \flag[m]{U})+\mu(\flag[1]{V}, \dots, \flag[m]{V}).
\]
To see the reverse inequality, note that if $T_U$ is a generating set for $\flag[1]{U}, \dots, \flag[m]{U}$ and $T_V$ is a generating set for 
$\flag[1]{V}, \dots, \flag[m]{V}$, then
\[
(T_U \times \{0\}^e) \cup (\{0\}^d \times T_V) \subseteq \mathbb{R}^{d+e}
\]
is a generating set for $\flag[1]{W}, \dots, \flag[m]{W}$ of size at most $|T_U|+|T_V|$.
\end{proof}
\begin{corollary}\label{cor:superadditivity}
The function $\mu$ is superadditive in its second argument, i.e., $$\mu(m, d + e) \geq \mu(m, d) + \mu(m, e)$$ for all $m, d, e \in \N$.
\end{corollary}
The case $m = d = e = 3$ of this corollary, together with \cref{prop:generic}, gives
\[
\mu(3, 6) \geq 2\mu(3, 3) \geq 2\floor{3 \cdot 3/2} = 10.
\]
By contrast, we saw in \cref{prop:generic} that a generic triple of flags in $\mathbb{R}^6$ can be generated by a set of only
$\ceil{3 \cdot 6/2} = 9$ vectors. This example shows that generic tuples do not in general achieve the maximum value $\mu(m,d)$ for odd $m$.

In a different direction, we next observe that $\mu$ is subadditive in its first argument, a fact that we already used near the end of the previous section.

\begin{proposition}\label{prop:splitflags}
Let $m, n, d \in \N$.
Suppose $\flag[1]{U}, \dots, \flag[m]{U}$ and
$\flag[1]{V}, \dots, \flag[n]{V}$ be tuples of flags in $\R^{d}$.
Then
\[
\mu(\flag[1]{U}, \dots, \flag[m]{U}, \flag[1]{V}, \dots, \flag[n]{V}) \leq \mu(\flag[1]{U}, \dots, \flag[m]{U})+\mu(\flag[1]{V}, \dots, \flag[n]{V})
\]
\end{proposition}
\begin{proof}
To generate
$\flag[1]{U}, \dots, \flag[m]{U}, \flag[1]{V}, \dots, \flag[n]{V}$, simply combine a generating set for
$\flag[1]{U}, \dots, \flag[m]{U}$ of size
$\mu(\flag[1]{U}, \dots, \flag[m]{U})$
and a generating set for 
$\flag[1]{V}, \dots, \flag[n]{V}$
of size 
$\mu(\flag[1]{V}, \dots, \flag[n]{V})$.
\end{proof}
\begin{corollary}\label{cor:subadditivity}
The function $\mu$ is subadditive in its first argument, i.e., 
$${\mu(m + n, d)} \leq {\mu(m, d)} + {\mu(n, d)}$$ for all $m, n, d \in \N$.
\end{corollary}

The only remaining ingredient for the proof of \cref{thm:main} is the following result, which is the most difficult part of this paper.
\begin{theorem}\label{thm:3flags}
We have $\mu(3, d) \leq 5d/3$, i.e., any triple of flags in $\R^d$ admits a simultaneous generating set of size 
at most ${5d/3}$.
\end{theorem}
We will prove \cref{thm:3flags} in the next section. For now, we take it for granted and deduce \cref{thm:main}.
\begin{proof}[Proof of \cref{thm:main}]
There is nothing to prove when $m = 1$ or $d = 1$.
For even $m$, we can combine the lower bound
$\mu(m, d) \geq {md/2}$
from \cref{prop:generic} with the upper bound $\mu(m, d) \leq m/2 \cdot \mu(2, d) = md/2$ from \cref{cor:subadditivity} and \cref{thm:2flags}.

It remains to consider the case where $m \geq 3$ is odd and $d \geq 2$. We provisionally set
\[
\lambda(m, d) \defeq \frac{md}2 + 
\floor{\frac{2d}3} -\frac d2.
\]
To show that $\mu = \lambda$, we will establish the inequalities $\mu \geq \lambda$ and $\mu \leq \lambda$ separately. 
For now, observe that 
\begin{align*}
\lambda(m+2, d)-\lambda(m, d) = d \quad\text{and}\quad
\lambda(m, d+3) - \lambda(m, d) = \left\lceil \frac{3m}{2}\right\rceil.
\end{align*}
We start with the lower bound on $\mu$.  The generic construction in \cref{prop:generic} gives
\begin{equation*}
\mu(m,d) \ge \left\lceil \frac{md}2 \right\rceil = \lambda(m, d) \quad \text{for $2\le d\le 4$.}
\end{equation*}
Thanks to \cref{cor:superadditivity}, we know that $\mu(m, d)$ is superadditive in $d$ and thus
$$ \mu(m, d+3) - \mu(m, d) \ge \mu(m, 3) \ge \left\lceil \frac{3m}2 \right\rceil = \lambda(m, d+3) - \lambda(m, d). $$
It follows by induction on $d$ that $\mu(m, d)\ge \lambda(m, d)$.

We now treat the upper bound on $\mu$. 
\cref{thm:3flags} tells us that
\begin{equation*}
\mu(3,d) \le \left\lfloor \frac{5d}3 \right\rfloor = \lambda(3, d) \,.
\end{equation*}
For $m \geq 5$, we have
$\mu(m - 3, d) = (m - 3)d/2$ since $m - 3$ is even.
Also, thanks to \cref{cor:subadditivity}, we know that $\mu(m, d)$ is subadditive in $m$ and thus
$$ 
\mu(m, d) \leq \mu(m - 3, d) + \mu(3, d) \leq \frac{(m - 3)d}{2} + \floor{\frac{5d}{3}} = \lambda(m, d),
$$
as desired.
\end{proof}

\section{Three flags}\label{sec:3-flags}
In this section 
we prove \cref{thm:3flags}, the core
result of this paper.
Fix $d\in \N$ and a triple of flags $\flag{U},\z\flag{V},\z\flag{W}$ in $\R^d$.  We wish to show that $\mu(\flag U, \flag V, \flag W) \leq 5d/3$.

Say that a set of layers is \emph{compatible} if there is a single vector which is new for all of them.  Every singleton set is compatible, and some sets of size $2$ or $3$ are compatible.  There can be no larger compatible sets, since a compatible set contains at most one layer from each of our three flags.  The inequality $\mu(\flag U, \flag V, \flag W) \leq 5d/3$ is equivalent to the existence of a family of at most $5d/3$ compatible sets whose union covers all of $\flag U, \flag V, \flag W$ (excluding $0$-dimensional layers).  We will see that compatible pairs can be neatly characterized, whereas compatible triples are somewhat more elusive.

Recall the permutations $\perm{U}{V},\z \perm{V}{W},\z\perm{W}{U} \in \symmetric{d}$ provided by \cref{thm:2flags}.

\subsection{A tale of two graphs}\label{subsec:graphs}
We introduce two auxiliary graphs $G$, $\widetilde G$ associated to the triple of flags $\flag U$, $\flag V$, $\flag W$.  The vertex set of each graph is the set of formal symbols
\[
\Gamma:=\{\flag{U}[1], \dots, \flag{U}[d], \flag{V}[1], \dots, \flag{V}[d], \flag{W}[1], \dots, \flag{W}[d]\}.
\]
We will not carefully distinguish between these vertices and the corresponding subspaces of $\mathbb{R}^d$; the resulting mild abuses of notation should not cause any confusion. 

The edges of our graphs will encode intersection properties of pairs of layers in our flags:
\begin{itemize}
\item We include the edge $(\flag{U}[i], \flag{V}[j])$ in $G$ if and only if $$\perm{U}{V}(i) = j.$$  Recall from the proof of \cref{thm:2flags} that $\perm{U}{V}(i)$ is the smallest $j \in [d]$ such that $(\flag{U}[i]\setminus \flag{U}[i-1])\cap \flag{V}[j]\neq\varnothing$. Likewise, ${\perm[-1]{U}{V}}(j)$ is the smallest $i \in [d]$ such that $(\flag{V}[j]\setminus \flag{V}[j-1])\cap \flag{U}[i]\neq\varnothing$. 
The edges of the form $(\flag{V}[j], \flag{W}[k])$ and $(\flag{W}[k], \flag{U}[i])$ are determined analogously.
\item We include the edge $(\flag{U}[i], \flag{V}[j])$ in $\widetilde G$ if and only if $${(\flag{U}[i]\setminus \flag{U}[i-1])\cap (\flag{V}[j]\setminus \flag{V}[j-1])\neq\varnothing},$$ i.e., if and only if the pair $\{\flag{U}[i], \flag{V}[j]\}$ is compatible. Again, edges $(\flag{V}[j], \flag{W}[k])$ and $(\flag{W}[k], \flag{U}[i])$ are determined analogously.
\end{itemize}
\begin{figure}
\centering
\includegraphics[scale=1.1]{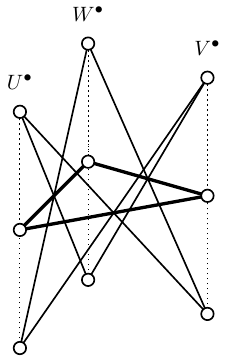} \qquad \qquad \qquad
\includegraphics[scale=1.1]{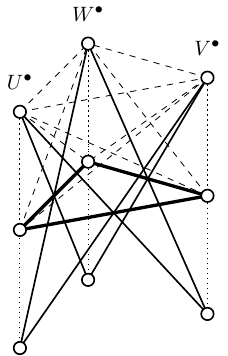}
\caption{The graphs $G$ (left) and $\widetilde G$ (right) for a triple of generic flags in $\mathbb{R}^3$.  The graph $G$ consists of a triangle (in bold) and a $6$-cycle; these cycles cross.  The resulting edges present in $\widetilde G$ but not in $G$ are drawn dashed.}
\label{fig:generic}
\end{figure}
Both $G,\widetilde G$ are tripartite graphs, with partite sets 
given by $\flag U, \flag V, \flag W$.  The graph $G$ is moreover an edge-disjoint union of three perfect matchings: one between $\flag U, \flag V$, one between $\flag V, \flag W$, and one between $\flag W, \flag U$. It follows that $G$ is a disjoint union of cycles, each with an equal number of vertices in $\flag U,\flag V,\flag W$. In particular, each cycle has length a multiple of~3.

It will be useful for us to depict the graphs $G, \widetilde G$ pictorially.  We place the vertices of $\flag U, \flag V, \flag W$ in three vertical columns, one for each flag.
In each column we arrange vertices by dimension, with
$U^i$ being the $i$-th vertex from the bottom.
The three columns 
form parallel edges of a right triangular prism.  Now each edge of $G$ or $\widetilde G$ can be depicted as a line segment  on one of the rectangular faces of 
this prism.  See \cref{fig:generic} for an example.

We say that two edges \emph{cross} if they intersect in this geometric representation.
Formally, edges $(U^i, V^j)$ and $(U^{i'}, V^{j'})$ cross if $i > i'$ and $j < j'$ or vice versa.
We say that two cycles of $G$ \emph{cross} if an edge of one crosses an edge of the other.  If two cycles are non-crossing, then one cycle must lie completely below the other; thus any collection of pairwise non-crossing cycles is totally ordered by height.

It is clear that the edge set of $\widetilde G$ contains the edge set of $G$. 
Less obviously,
the edge set of $G$ 
completely determines the  edge set of $\widetilde G$, according to the following rule.

\begin{lemma}\label{lem:tildeG}
Let $i,j \in [d]$. Let $i', j' \in [d]$ be the unique indices such that $(\flag{U}[i], \flag{V}[j'])$ and $(\flag{U}[i'], \flag{V}[j])$ are edges of $G$.
Then $(\flag{U}[i],\flag{V}[j])$ is an edge in $\widetilde G$ if and only if $i\ge i'$ and $j\ge j'$.
\end{lemma}
Pictorially, this lemma says that $\widetilde G$ can be obtained from $G$ by adding an edge between the two top vertices in each pair of crossing edges; see \cref{fig:crossing-edges}.  Of course, the lemma also holds with $(U,\z V)$ replaced by $(V,\z W)$ or $(W,\z U)$.

\begin{figure}[h]
\centering
\includegraphics{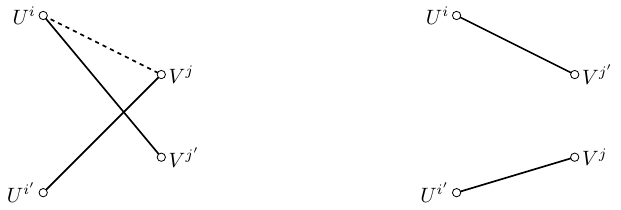}
\caption{An illustration of \cref{lem:tildeG}.  The left part of the diagram shows a pair of crossing edges in $G$, with the corresponding additional edge (drawn dotted) in $\widetilde G$.  The right part of the diagram shows a pair of non-crossing edges in $G$, for which there is no additional edge in $\widetilde G$.}
\label{fig:crossing-edges}
\end{figure}

\begin{proof}[Proof of \cref{lem:tildeG}]
Recall that $j'$ is minimal in $[d]$ such that ${(\flag{U}[i]\setminus \flag{U}[i-1])\cap \flag{V}[j'] \neq \varnothing}$.  It follows that if $j < j'$, then
\[
(\flag{U}[i]\setminus \flag{U}[i-1])\cap (\flag{V}[j]\setminus \flag{V}[j-1]) \subseteq (\flag{U}[i]\setminus \flag{U}[i-1])\cap \flag{V}[j] = \varnothing
\]
and thus $(\flag{U}[i],\flag{V}[j])$ is not an edge in $\widetilde G$.  Likewise, $(\flag{U}[i],\flag{V}[j])$ is not an edge in $\widetilde G$ if $i<i'$.
Suppose now that $i \geq i'$ and $j \geq j'$.  The characterizations of $j',i'$ ensure that $$(\flag{U}[i]\setminus \flag{U}[i-1]) \cap \flag{V}[j] \quad \text{and} \quad \flag{U}[i]\cap (\flag{V}[j]\setminus \flag{V}[j-1])$$ are both non-empty, i.e., $\flag{U}[i] \cap \flag{V}[j]$ is not contained in either of $\flag{U}[i-1]$ or $\flag{V}[j-1]$.  
Hence, $\flag{U}[i] \cap \flag{V}[j]$ is not contained in the union $\flag{U}[i-1] \cup \flag{V}[j-1]$, i.e., ${(\flag{U}[i]\setminus \flag{U}[i-1])\cap (\flag{V}[j]\setminus \flag{V}[j-1])}$ is nonempty.  Thus $(\flag{U}[i], \flag{V}[j])$ is an edge in $\widetilde G$, as desired.
\end{proof}
We remark that for a triple of layers to be compatible it is necessary---but in general not sufficient---that the corresponding vertices form a triangle in $\widetilde G$.  

\subsection{Even cycles and crossing odd cycles}

Our goal is to show that $\Gamma$ can be covered by at most $5d/3$ compatible sets.  Our first step towards this goal is showing that certain subsets of $\Gamma$ can be covered efficiently by compatible pairs.

\begin{lemma}\label{lem:graphmatching}
Let $S \subseteq \Gamma$ be a set of vertices.
\begin{enumerate}[(i)]
\item If $S$ forms an even cycle in $G$, then $S$ can be covered by $\abs{S}/2$ compatible pairs.
\item If $S$ forms an odd cycle in $G$, then $S$ can be covered by $(\abs{S} + 1)/2$ compatible pairs.
\item If $S$ is the union of two crossing odd cycles in $G$, then $S$ can be covered by $\abs{S}/2$ compatible pairs.
\end{enumerate}
\end{lemma}
\begin{proof}
The definition of $\widetilde G$ ensures that the number of compatible pairs needed to cover $S$ is precisely the edge cover number of the subgraph of $\widetilde G$ induced by $S$.  The first two parts of the lemma follow quickly from the fact that $G$ is a subgraph of $\widetilde G$.

We now treat the third part of the lemma.  Since the two cycles are crossing, \cref{lem:tildeG} implies that there is an edge in $\widetilde G$ connecting the two cycles; it follows that $S$ can be covered by $\abs{S}/2$ edges.  See \cref{fig:crossing-cycles}.
\end{proof}

\begin{figure}
\centering
\includegraphics[scale=0.9]{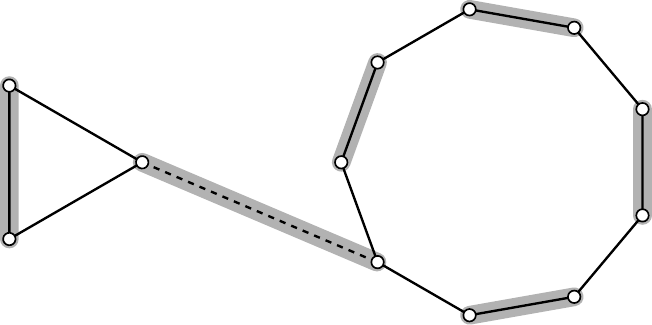}
\caption{The edge cover (shaded) used in the third part of \cref{lem:graphmatching}.}
\label{fig:crossing-cycles}
\end{figure}

\subsection{Grouping cycles by efficiency}
Recall that $G$ is a disjoint union of cycles with lengths divisible by $3$.   We divide these cycles into three groups $A,B,C$ according to how efficiently we can cover them with compatible sets, as follows.

First, put all of the even cycles in $A$, thereby removing them from consideration.
Do the same with each triangle in $G$ that forms a compatible triple. 
Then, greedily, whenever we discover among the remaining odd cycles two that cross we place both of them in $A$ and remove them from consideration.

\begin{figure}[b]
\centering
\includegraphics[scale=1.2]{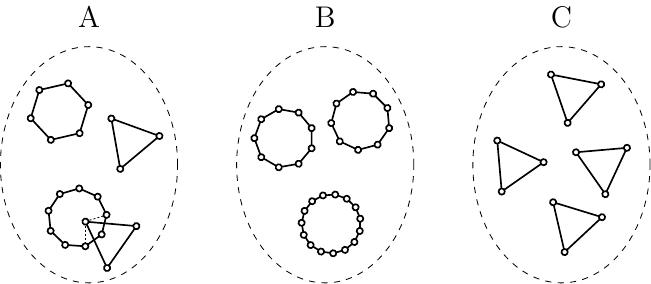}
\caption{A schematic of the types of cycles appearing in the groups $A,\z B,\z C$.  In $A$, the isolated triangle represents a compatible triple, and the two odd cycles at the bottom are crossing (with the dotted edges representing edges present in $\widetilde G$ but not $G$).}
\label{fig:buckets}
\end{figure}

At the end, we are left with some number of non-crossing odd cycles.
Let $C$ consist of the triangles among them, and let $B$ consist of the other cycles.
Thus each cycle in $B$ has length at least $9$, and no triangle in $C$ forms a compatible triple.  Note that the cycles in $B$ and $C$ are ordered by height.  See \cref{fig:buckets}.

Let $|A|$,~$|B|$,~$|C|$ denote the total number of \emph{vertices} in the cycles in $A$,~$B$,~$C$, respectively.
Clearly
\begin{equation}\label{eq:abc1}
|A|+|B|+|C|= \abs{\Gamma} = 3d.
\end{equation}
Moreover, \cref{lem:graphmatching} ensures that the following hold:
\begin{itemize}
\item Each even cycle $S$ in $A$ can be covered by 
$\abs{S}/2$ compatible sets. 
\item Each triangle $S$ in $A$ which forms a compatible triple can be covered by $1 < \abs{S}/2$ compatible sets. 
\item Each union $S$ of two crossing odd cycles 
in $A$ can be covered by
$\abs{S}/2$ compatible sets.
\item Each cycle $S$ in $B$ can be covered by 
$(\abs{S} + 1)/2\le 5/9 \cdot \abs{S}$ compatible sets.
\item Each triangle $S$ in $C$ can be covered by $2=2/3 \cdot |S|$ compatible sets.
\end{itemize}
These observations together imply that the vertex set $\Gamma$ can be covered by at most
\begin{equation}\label{eq:abc2}
\frac12|A| + \frac59|B|+\frac23|C|
\end{equation}
compatible sets.  It remains to control the sizes of $A,B,C$.

\subsection{A monovariant that controls triangles}
Recall that our goal is to cover $\Gamma$ by at most $5d/3$ compatible sets.  Since $|\Gamma|=3d$, our goal is to cover $\Gamma$ at an average ``rate'' of at most $5/9$ of a compatible set per vertex.  We see from \eqref{eq:abc2} that vertices in $A$ can be covered more efficiently than this critical rate, vertices in $B$ can be covered at precisely the critical rate, and
our procedure for covering vertices in $C$ is less efficient than the critical rate.  One should therefore think of $A$,~$B$,~$C$ as being good, neutral, and bad, respectively, and we would like to know that $A$ is large compared to $C$.  More precisely, in order to 
bound \eqref{eq:abc2} from above by $5d/3$, we need
$$|A| \geq 2|C|.$$
Since $|A|+|B|+|C|=3d$ by \eqref{eq:abc1}, this inequality is equivalent to
\begin{equation}\label{eq:abc3}
\frac12|A| + \frac13|B| \ge d,
\end{equation}
which will be more convenient for us.
Establishing \eqref{eq:abc3} is the most difficult part of our proof. 

In order to prove \eqref{eq:abc3}, we will study the evolution of the quantity $$\dim(i, j, k)\defeq \dim(\flag{U}[i]\cap \flag{V}[j]\cap \flag{W}[k])$$ as we travel from $(0,0,0)$ to $(d,d,d)$ by incrementing one coordinate at a time according to a particular procedure.  
One should think of this procedure as a lattice walk from $(0,0,0)$ to $(d,d,d)$ using the steps $(1,0,0),(0,1,0),(0,0,1)$.  We will say that a step from the position 
$(i-1, j, k)$ (respectively, $(i,j-1,k)$ or $(i,j,k-1)$) to the position $(i, j, k)$ is a \emph{hop} on the vertex $\flag{U}[i]$ (respectively, $\flag{V}[j]$ or $\flag{W}[k]$) at $(i, j, k)$. Over the course of our procedure, we hop exactly once on each vertex of $\Gamma$. We define the \emph{cost} of the vertex $\flag{U}[i]$ to be $$\cost(\flag{U}[i]):=\dim(i, j, k) - \dim(i-1, j, k),$$ where $j,k$ are such that we hopped on $\flag{U}[i]$ at $(i, j, k)$.  The cost of a vertex $\flag{V}[j]$ or $\flag{W}[k]$ is defined analogously. The cost of each vertex is either $0$ or $1$. 
We will say that the cost of a set of vertices is the sum of the costs of its elements. Observe that
\[
\cost(\Gamma)=\dim(d,d,d)-\dim(0,0,0)=d-0=d.
\]
It is crucial to our strategy that we are free to choose the lattice path from $(0,0,0)$ to $(d,d,d)$. A suitable choice will allow us to control the costs of the cycles in each of $A,B,C$.
\begin{lemma}\label{lem:bound_cost}
Let $A,B,C$ be as above.  Then there exists a lattice path from $(0,0,0)$ to $(d,d,d)$
for which the following holds:
\begin{enumerate}
\item The cost of any cycle $S$ in $A$ is at most $\abs{S}/2$.
\item The cost of any cycle $S$ in $B$ is at most $\abs{S}/3$.
\item The cost of any triangle in $C$ is $0$.
\end{enumerate}
\end{lemma}

For the lattice path produced by \cref{lem:bound_cost}, we have
$$d=\cost(\Gamma)=\cost(A)+\cost(B)+\cost(C) \leq \frac12 |A|+\frac13 |B|,$$
which is precisely \eqref{eq:abc3}.  We will prove \cref{lem:bound_cost} in the next subsection.

\subsection{Proof of the key lemma}
Before proving \cref{lem:bound_cost}, we isolate two technical linear-algebraic facts.

\begin{lemma}\label{lem:sleight-of-hand}
Let $1\le i,j, k, j'\le d$ be such that $\flag{U}[i]$ is adjacent to $\flag{V}[j']$ in $G$.
If $j <j'$, then $\flag{U}[i]\cap \flag{V}[j]\cap \flag{W}[k] = \flag{U}[i-1]\cap \flag{V}[j]\cap \flag{W}[k]$. Analogous statements hold for all permutations of $U, V, W$.
\end{lemma}
\begin{proof}
The hypothesis
$\sigma_{\flag U, \flag V}(i) = j' > j$ forces
$(\flag{U}[i]\setminus \flag{U}[i-1])\cap \flag{V}[j]=\varnothing$, i.e., $\flag{U}[i]\cap \flag{V}[j]=\flag{U}[i-1]\cap \flag{V}[j]$.
\end{proof}

\begin{lemma}\label{lem:real-magic}
If $\flag{U}[i],\,\flag{V}[j],\,\flag{W}[k]$ form a triangle in $G$ but do not constitute a compatible triple, then
\[
\flag{U}[i]\cap \flag{V}[j]\cap \flag{W}[k] = \flag{U}[i-1]\cap \flag{V}[j-1]\cap \flag{W}[k-1].
\]
\end{lemma}
\begin{proof}
To prove the lemma, it is enough to show that the identity
\[
(\flag{U}[i]\cap \flag{V}[j]\cap \flag{W}[k]) \setminus (\flag{U}[i-1]\cap \flag{V}[j-1]\cap \flag{W}[k-1]) = 
(\flag{U}[i] \setminus \flag{U}[i-1])\cap (\flag{V}[j] \setminus \flag{V}[j-1]) \cap (\flag{W}[k] \setminus \flag{W}[k-1])
\]
holds whenever
$\flag{U}[i],\,\flag{V}[j],\,\flag{W}[k]$ form a triangle in $G$.
It is easy to see that the left-hand side contains the right-hand side. For the reverse inclusion, it suffices (by symmetry) to show that 
\[
(\flag{U}[i]\cap \flag{V}[j]\cap \flag{W}[k]) \setminus (\flag{U}[i-1]\cap \flag{V}[j-1]\cap \flag{W}[k-1]) \subseteq 
\flag{U}[i] \setminus \flag{U}[i-1],
\]
or, equivalently, that
\[
\flag{U}[i-1]\cap \flag{V}[j]\cap \flag{W}[k] \subseteq \flag{U}[i-1]\cap \flag{V}[j-1]\cap \flag{W}[k-1] \,.
\]
Since $\flag{V}[j]$ and $\flag{W}[k]$ are both adjacent to $\flag{U}[i]$ in $G$, we can apply \cref{lem:sleight-of-hand} twice (at the positions ${(i-1,j,k)}$ and ${(i-1,j-1,k)}$) to obtain
\[
\flag{U}[i-1]\cap \flag{V}[j]\cap \flag{W}[k]
= 
\flag{U}[i-1]\cap \flag{V}[j-1]\cap \flag{W}[k]
=
\flag{U}[i-1]\cap \flag{V}[j-1]\cap \flag{W}[k-1],
\]
as desired.
\end{proof}

\begin{proof}[Proof of Lemma \ref{lem:bound_cost}]
\Cref{lem:sleight-of-hand} tells us that if $(x,y)$ is an edge in $G$ and our lattice path hops on the vertex $x$ before the vertex $y$, then $\cost(x)=0$.  It follows that, for any choice of the lattice path, the set of vertices with cost $1$ forms an independent set in $G$.  This observation immediately implies the first part of \cref{lem:bound_cost}.

For the second and third parts of \cref{lem:bound_cost}, we must specify a choice of lattice path.  
Recall that the triangles in $C$ are ordered by height and 
that
each cycle in $B$ lies entirely between two such triangles
(or possibly below the bottom-most triangle or above the top-most triangle).  Let $M$ denote the number of triangles in $C$.  Then there are sequences of indices $i_1<\cdots<i_M$, $j_1<\cdots<j_M$, and $k_1<\cdots<k_M$
such that the triangles in $C$ are given by
$$
(\flag{U}[i_m], \flag{V}[j_m], \flag{W}[k_m])
$$
for $1 \leq m \leq M$.

\begin{figure}
\centering
\includegraphics{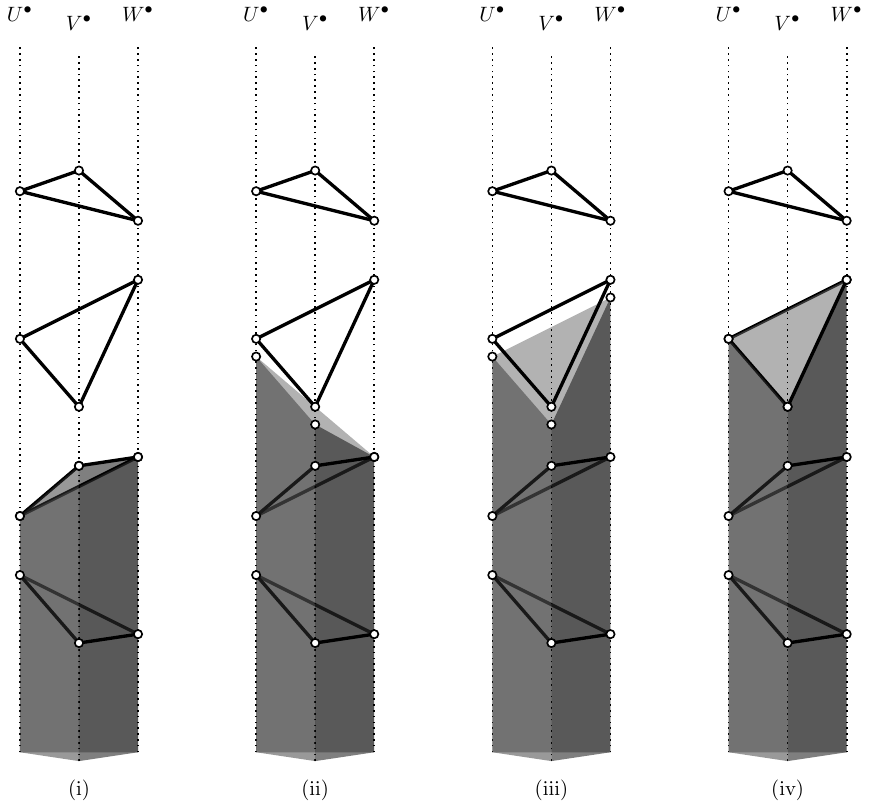}
\caption{Schematic ``snapshots'' (from left to right) of four stages of the hopping procedure from the proof of \cref{lem:bound_cost}.  The triangles in $C$ are drawn in bold.  In each snapshot the already-hopped-on vertices are those in the shaded region. (i) We have just hopped on the vertices of the second triangle from the bottom.  (ii) We have hopped on the vertices of $\flag{U},\flag{V}$ strictly between the second and third triangles.  (iii) We have hopped on the vertices of $\flag{W}$ strictly between the second and third triangles.  (iv) We have hopped on the vertices of the third triangle.}
\label{fig:hops}
\end{figure}

Now consider the following lattice path. We begin by hopping in the order
$$\flag{U}[1], \flag{U}[2], \ldots, \flag{U}[i_1-1], \quad \flag{V}[1], \flag{V}[2], \ldots, \flag{V}[j_1-1], \quad \flag{W}[1], \flag{W}[2], \ldots, \flag{W}[k_1-1]$$
and then hop over the triangle $\flag{U}[i_1], \flag{V}[j_1], \flag{W}[k_1]$. We continue with
$$\flag{U}[i_1+1], \ldots, \flag{U}[i_2-1],\quad \flag{V}[j_1+1], \ldots, \flag{V}[j_2-1],\quad \flag{W}[k_1+1], \ldots, \flag{W}[k_2-1],$$
and so on, until we hop over the last triangle 
$\flag{U}[i_M], \flag{V}[j_M], \flag{W}[k_M]$ and conclude with
$$\flag{U}[i_M+1], \ldots, \flag{U}[d],\quad \flag{V}[j_M+1], \ldots, \flag{V}[d],\quad \flag{W}[k_M+1], \ldots, \flag{W}[d].$$
In other words, for each triangle in $C$, we increment $i$ until we reach the triangle, then we increment $j$ until we reach the triangle, then we increment $k$ until we reach the triangle, and finally we hop on the vertices of the triangle.

Since the cycles in $B \cup C$ are ordered by height, each cycle $S$ in $B$ is entirely contained in one of the chunks
$$\flag{U}[i_m+1], \ldots, \flag{U}[i_{m+1}-1],\quad \flag{V}[j_m+1], \ldots, \flag{V}[j_{m+1}-1],\quad \flag{W}[k_m+1], \ldots, \flag{W}[k_{m+1}-1]$$
for some $0 \leq m \leq M$ (with the conventions ${i_0=j_0=k_0:=0}$ and $i_{M+1}=j_{M+1}=k_{M+1}:=d+1$).  As in the first paragraph of the proof, it follows from \cref{lem:sleight-of-hand} that each vertex of $S$ in $\flag{U},\flag{V}$ has cost $0$, because such a vertex is the endpoint of an edge of $G$ whose other endpoint is a not-yet-hopped-on vertex in $\flag{W}$.  Since $S$ has $|S|/3$ vertices in $\flag{W}$, we conclude that $\cost(S) \leq |S|/3$.  This establishes the second part of \cref{lem:bound_cost}.

\cref{lem:real-magic} tells us precisely that the cost of each triangle $(\flag{U}[i_m], \flag{V}[j_m], \flag{W}[k_m])$ is zero.  This proves the third part of \cref{lem:bound_cost}.
\end{proof}

\subsection{Conclusion}
We are now ready to assemble the pieces. We emphasize that multiple striking numerical coincidences are necessary for the proof to work.
\begin{proof}[Proof of \cref{thm:3flags}]
Let $\flag U$, $\flag V$, $\flag W$ be flags in $\mathbb{R}^d$.  Construct the graphs $G,\widetilde G$ and the sets $A,B,C$ as above.  \cref{lem:bound_cost} implies the crucial inequality \eqref{eq:abc3}.  Combining this with \eqref{eq:abc1} and \eqref{eq:abc2}, we can upper-bound $\mu(\flag U, \flag V, \flag W)$ by
$$\frac12|A|+\frac59|B|+\frac23|C|
=
\frac23\big(|A|+|B|+|C|\big) - \frac13\left(\frac12|A|+\frac13|B|\right)
\le \frac233d - \frac13d = \frac53d,$$
as desired.
\end{proof}

\subsection{Equality cases}
\cref{thm:3flags} tells us that $\mu(\flag U, \flag V, \flag W)\leq 5d/3$ for all triples of flags $\flag U, \flag V, \flag W$ in $\mathbb{R}^d$, and it is natural to ask when equality is achieved.
Clearly, we must have $d$ a multiple of 3.
A close examination of the arguments in this section (including the consideration of more intricate lattice paths than the one used in the proof of \cref{lem:bound_cost}) reveals that the equality $\mu(\flag U, \flag V, \flag W)=5d/3$ 
cannot hold unless the associated graph $G$ satisfies the following constraints:
\begin{itemize}
\item $G$ consists of exactly $d/3$ triangles, none of them forming a compatible triple, and some number of even cycles;
\item the triangles in $G$ are non-crossing and hence can be ordered by height;
\item for any two consecutive edges in an even cycle, there is some triangle which both edges cross.\footnote{In other words, as we travel around an even cycle, the edges alternately cross triangles from below and from above.}
\end{itemize}
A separate clever argument then shows that these conditions actually \emph{suffice} to guarantee equality.
We leave the (non-obvious) details of these deductions to the adventurous reader.

\begin{figure}[h]
    \centering
    \includegraphics[scale=1.3]{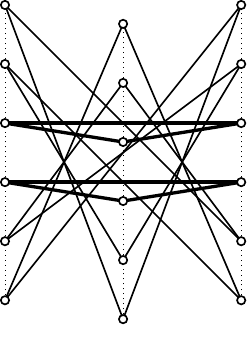}
    \qquad 
    \includegraphics[scale=0.8]{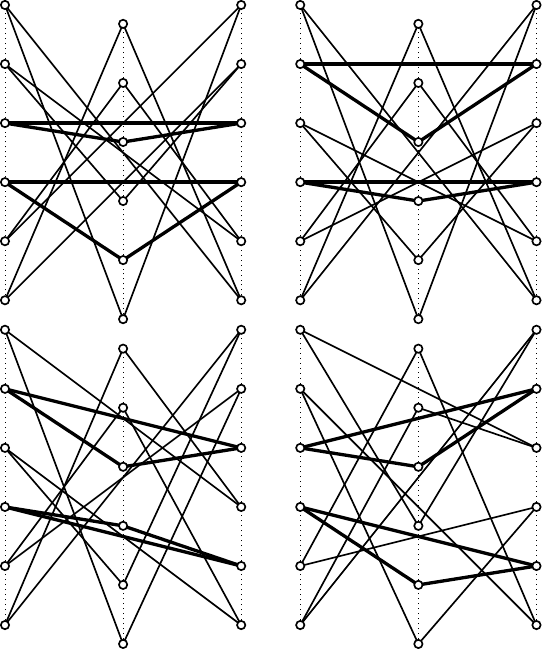}
    \caption{The graph $G$ for several equality cases.
    The large graph on the left is a particularly symmetric example which contains a $12$-cycle.  The four smaller graphs on the right are randomly-chosen examples.}
    \label{fig:equality}
\end{figure}

The previous paragraph gave an elaborate combinatorial characterization of the equality cases for \cref{thm:3flags} in terms of the auxiliary graph $G$.
The proof of \cref{prop:sum_flags}, by contrast, supplies us with a simple explicit construction of some equality cases: Take direct sums of the configuration of three generic flags in $\R^3$.  The graph $G$ for this ``building block'' is depicted in \cref{fig:generic}.
One might wonder if all the equality cases can be obtained via simple modifications of the construction of the previous paragraph, in which case our combinatorial characterization could be simplified.

This hope is dashed when one examines the hundreds of different equality cases for $m = 3$, $d = 6$ 
(see \cref{fig:equality} for a sampling).\footnote{See \cref{sec:EL} for an indication of how to generate all possible graphs $G$.}
In some instances the graph $G$ contains $12$-cycles. This diversity of equality cases suggests that a full geometric understanding of them
would be a difficult task, and it underscores that something like the delicate combinatorics of \cref{sec:3-flags} is necessary for the proof of \cref{thm:3flags}.

\section{Concluding remarks}\label{sec:concluding}

We conclude with some remarks and open problems.

\subsection{Grids of dimensions of intersections}\label{sec:EL}

Given an $m$-tuple of flags in $\mathbb{R}^d$, one can form a $d \times \cdots \times d$ grid recording the dimensions of all $m$-fold intersections of layers of flags (one layer from each flag).
The data in this grid is sufficient to determine whether a set of layers is compatible (i.e., whether there is a single vector that is new for all of the layers)
and, consequently, is enough to determine the value of $\mu$.\footnote{Here we are implicitly using the fact that no subspace of $\mathbb{R}^d$ can be covered by a finite union of proper subspaces; for the situation over finite fields, see the following subsection.}  This is more or less the perspective taken in the proofs of \cref{sec:2flags,sec:generic,sec:3-flags}.

It is thus natural to ask for a combinatorial characterization of the set of all possible such grids of dimensions of intersections for given values of $m,d$.  The Bruhat decomposition provides such a characterization in the case of pairs of flags ($m=2$).

Eriksson and Linusson \cite{EL,EL2} gave a simple necessary condition, valid for all $m$, for a grid of numbers to be the grid of dimensions of intersections of an $m$-tuple of flags.  An older result of Shapiro, Shapiro, and Vainshtein \cite{SSV} shows somewhat surprisingly, that this condition is also sufficient in the case $m=3$.  This characterization allowed us to obtain computational data for triples of flags with small $d$, which was quite useful for the investigations leading to the present paper.

The case $m>3$ is less well understood.  Motivated by examples of non-representable matroids, Billey and Vakil \cite{BV} have shown that Eriksson and Linusson's necessary condition is \emph{not} sufficient for 
$m \geq 4$.
See \cite{BV} for further discussion and references.

\subsection{Finite fields}
Our main problem can be posed over fields other than $\mathbb{R}$. More precisely, for a field $\mathbb{K}$ and positive integers $m, d$, let $\mu_{\K}(m, d)$ be the smallest natural number $n$ such that any $m$-tuple of flags in $\K^d$ can be generated by at most 
$n$ vectors.
As mentioned in the introduction, the proof presented in this paper shows that $\mu_{\K} = \mu_{\R} = \mu$ for all infinite fields $\K$.
In fact, the proof of our upper bound on $\mu$ works for all fields, so we always have $\mu_{\K} \leq \mu$.

As for lower bounds, it is not hard to see\footnote{by general model-theoretic considerations, or a direct counting argument} that for fixed $m$ we can find $m$ transverse flags in each of $\K^2$ and $\K^3$ as soon as $\K$ is large enough in cardinality.
Repeating the proof of~\cref{thm:main}, we find that $\mu_{\mathbb{K}}(m, d) = \mu(m, d)$ for all sufficiently large fields $\mathbb{K}$ (uniformly in $d$).  At the same time, one cannot expect equality to hold for all $\mathbb{K},m,d$: When $\mathbb{K}$ is finite,
choosing a nonzero vector in each 1-dimensional subspace of $\K^d$ gives a universal generating set, independent of $m$.
A careful analysis, based in part on~\cite[Table 6.1]{Ox}, reveals 
the following two propositions.
\begin{proposition}
Fix $m \in \mathbb{N}$ and a prime power $q\ge 2$. 
We have $\mu_{\F_q}(m, d) = \mu(m, d)$ for all $d \in \mathbb{N}$ if and only if $q\ge m-1$.
\end{proposition}

\begin{proposition}
    Fix a field $\K$ and $m \in \mathbb{N}$. 
    Then $\mu_{\K}(m, d) \geq md/2$ for all $d\in\N$ sufficiently large.
    In particular, when $m$ is even, we have $\mu_{\K}(m, d) = \mu(m, d)$ for all sufficiently large $d$.
\end{proposition}
With these facts in mind, we pose the following open problem.
\begin{problem}
    Fix a field $\K$ and $m\in\mathbb N$.
    Do we always have $\mu_\K(m, d) = \mu(m, d)$ for all $d\in\N$ sufficiently large?
    The simplest open case here is $m = 5$ and $\K = \F_2$.
\end{problem}

\subsection{Matroid analogues}
Problems in linear algebra can often be extended to the realm of matroids.
The analogues for matroids of ``vector'', ``span'', ``subspace'', ``hyperplane'' and ``dimension'' are ``ground set element'', ``span'', ``flat'', ``hyperplane'' and ``rank'' (see, e.g., \cite{Ox}). Using this dictionary one can define ``(complete) flags of flats'' in a matroid.
The naive matroid formulation for the main problem studied in this paper would ask for the minimal number of ground set elements needed to simultaneously span each of the flats in each of $m$ complete flags of flats in a matroid of rank $d$.
Unfortunately, this approach fails to get off the ground because the analogue of \cref{thm:2flags} is false for matroids.
(This is because intersections of flats can behave quite badly, e.g., an intersection of two rank-$10$ flats in ambient rank $12$ can have rank $1$.)
Rather, the \emph{correct} formulation is based on a dualized version of \cref{thm:2flags}.
Given $m$ flags of flats in a matroid, we ask for the minimum number of hyperplanes needed to write each flat as an intersection of (some of these) hyperplanes.
Note that for vector spaces, this question is equivalent to the question that we have considered in this paper.
We ask if $\mu(m,d)$ remains an upper bound in the matroid setting.
\begin{question}
Let $m,d \in \mathbb{N}$.  Is it true that for every rank-$d$ matroid $M$ and every $m$-tuple of complete flags of flats in $M$, there is a set of $\mu(m,d)$ hyperplanes in $M$ that generates (via intersections) all of the flags?  If not, then how many hyperplanes (as a function of $m,d$) are needed in the worst case?
\end{question}
In this dualized setting, the analogue of \cref{thm:2flags} goes through and gives a satisfactory answer for $m = 2$.  A suitable version of our crucial \cref{lem:real-magic} also holds. Unfortunately, the analogue of \cref{lem:sleight-of-hand} is false for general matroids (essentially because the entire ground set can be the union of two proper flats). This problem disappears if we restrict our attention to a special class of matroids dubbed \emph{unbreakable matroids}\footnote{Note that the matroids induced by vector spaces (namely, projective spaces) are
unbreakable.} by Pfeil in his doctoral thesis~\cite{Pf}; see also \cite{OP}. This class
does obey an appropriate version of \cref{lem:sleight-of-hand}; it follows that our upper-bound arguments (in particular \cref{thm:3flags}) extend to this class of matroids.

In general the matroid version of our main problem remains open.  We remark, however, that if 
a counterexample to the matroid version of \cref{thm:3flags} exists, then the task of finding one seems computationally difficult.

\section*{Acknowledgements}
We have benefited from helpful discussions with many people, including Sara Billey, Matt Larson, Sam Mundy, Junyao Peng, and Stefan Steinerberger.  
The first author is supported by the National Science Foundation
under Grant No. DMS-1926686. The second author is supported in part by the NSF Graduate Research Fellowship Program under grant DGE–203965.


\begin{thebibliography}{99}

\bibitem{BV} S.\ Billey and R.\ Vakil, Intersections of Schubert varieties and other permutation array schemes. In \emph{Algorithms in Algebraic Geometry}, IMA Volumes in Mathematics and its Applications,
vol.\ 146 (ed.\ A.\ Dickenstein, F.-O.\ Schreyer, and A.\ Sommese), 2008, 21--54.

\bibitem{bourbaki} N.\ Bourbaki, \emph{Groupes et Alg\`ebres de Lie, Chapitres 4 \`a 6}.  \'El\'ements de Math\'ematique, Masson (Paris), 1981.

\bibitem{EL} K.\ Eriksson and S.\ Linusson, A combinatorial theory of higher-dimensional permutation arrays. \emph{Adv.\ Appl.\ Math.}, {\bf 25.2} (2000), 194--211.

\bibitem{EL2} K.\ Eriksson and S.\ Linusson, A decomposition of $\fl(n)^d$ indexed by permutation arrays. \emph{Adv.\ Appl.\ Math.}, {\bf 25.2} (2000), 212--227.

\bibitem{Ox} J.\ Oxley, \emph{Matroid Theory, Second Edition}.  Oxford Graduate Texts in Mathematics, Oxford University Press, 2011.

\bibitem{OP} J.\ Oxley and S.\ Pfeil, Unbreakable matroids. \emph{Adv.\ Appl.\ Math.}, {\bf 141} (2022), \#102404.

\bibitem{Pf} S.\ Pfeil, \emph{On Properties of Matroid Connectivity}. Ph.D.\ thesis, Louisiana State University, 2016.

\bibitem{SSV} B.\ Shapiro, M.\ Shapiro, and A.\ Vainshtein, On combinatorics and topology of pairwise intersections of Schubert cells in $\SL(n)/B$.  In \emph{The Arnol'd--Gelfand Mathematical Seminars}, Birkh\"auser (Boston), 1997, 397--437.

\bibitem{Wofsey} User Eric Wofsey, comment on the post ``How to construct a base compatible with two filtrations''.  Math StackExchange (2023), \url{https://math.stackexchange.com/q/4639005}.

\end{thebibliography}
\end{document}